\documentclass[12pt]{article}
\usepackage{a4, latexsym, amsfonts}
\usepackage{amssymb, amsmath, ifthen}
\usepackage[table]{xcolor}

\newcommand{\im}{\mathrm{im\, }}
\newcommand{\dom}{\mathrm{dom\, }}
\newcommand{\rank}{\mathrm{rank\, }}
\newcommand{\idrank}{\mathrm{idrank\, }}
\newcommand{\fix}{\mathrm{fix\, }}
\newcommand{\sh}{\mathrm{sh\, }}
\newcommand{\defect}{\mathrm{def\, }}

\newcommand{\qed}{\hfill $\square$ \vspace{3mm}}
\newcommand{\proof}{\noindent \textit{Proof. \ }}

\newtheorem{theorem}{Theorem}[section]

\newtheorem{lemma}[theorem]{Lemma}
\newtheorem{example}[theorem]{Example}

\newtheorem{corollary}[theorem]{Corollary}
\newtheorem{proposition}[theorem]{Proposition}

\date{}

\begin{document}

\title{On the semigroups of fence-decreasing and fence-preserving transformations on a finite fence}

\author{Gonca Ay\i k,  Hayrullah Ay\i k, Ilinka Dimitrova, J\"{o}rg Koppitz$^{*}$
	\let\thefootnote\relax\footnote{\textbf{Corresponding author: J\"{o}rg Koppitz}, Institute of Mathematics and Informatics, Bulgarian Academy of Sciences, Sofia, Bulgaria, \\
		E-mail: koppitz@math.bas.bg \\
		\textbf{Gonca Ay\i k}, Department of Mathematics, \c{C}ukurova University, Adana, Turkey,\\
		E-mail: agonca@cu.edu.tr \\
		\textbf{Hayrullah Ay\i k}, Department of Mathematics, \c{C}ukurova University, Adana, Turkey,\\
		E-mail: hayik@cu.edu.tr \\
		\textbf{Ilinka Dimitrova}, Faculty of Mathematics and Natural Science, South-West University "Neofit Rilski", Blagoevgrad, Bulgaria,\\
		E-mail: ilinka\_dimitrova@swu.bg }}

\maketitle

\noindent \textbf{Abstract:} For a natural number $n$, a fence $[n]=\{1\prec 2\succ 3\prec 4\succ 5\prec \cdots n\}$ is a partial ordered set. A partial transformation $\alpha$ is called fence-decreasing if $x\alpha \preceq x$ for all $x$ in the domain of $\alpha$, and fence-preserving if $x\prec y$ implies $x\alpha \preceq y\alpha$ for all $x$ and $y$ in the domain of $\alpha$. In this paper, we consider the monoids $\mathcal{DF}_{n}$ ($\mathcal{PDF}_{n})$ of all fence-decreasing full (partial) transformations as well as the monoid $\mathcal{PCF}_{n}$ of all fence-preserving transformations of $\mathcal{PDF}_{n}$. For these three monoids and some of their ideals, we determine the unique minimal generating set. Moreover, we calculate the rank of $\mathcal{DF}_{n}$, $\mathcal{PDF}_{n}$, and $\mathcal{PCF}_{n}$. Additional, we provide several combinatorial results concerning these three monoids.

\vspace*{2mm}

\vspace*{2mm}

\section{Introduction}

Let $n$ be a natural number, and let $[n]=\{1,\ldots ,n\}$,\, $[n]^{o}= \{k\in [n] : k\mbox{ is odd}\}$ and $[n]^{e}= \{k\in [n] : k\mbox{ is even}\}$. For a subset $Y \subseteq [n]$, a mapping $\alpha :Y \rightarrow [n]$ is called a \emph{partial transformation}, where $Y$ is the \emph{domain} of $\alpha$, in symbols: $\dom(\alpha)$, and $\im(\alpha) =\{ x\alpha : x\in \dom(\alpha)\}$ is the \emph{image $($range$)$} of $\alpha$. For convenient, the notation $x\alpha$ implies $x\in \dom(\alpha)$.  We call $x\in [n]$ a \emph{fixed point} of $\alpha$ if $x\alpha =x$ and denote the set of all fixed points of $\alpha$ by $\fix(\alpha)$. Moreover, the \emph{shift set}, \emph{defect set} and \emph{kernel} of $\alpha$ are defined by $\sh(\alpha) =\im(\alpha) \setminus \fix(\alpha)$,\, $\defect(\alpha) =[n] \setminus \im(\alpha)$ and 
$$ker(\alpha) =\{ (x,y)\in \dom(\alpha)\times \dom(\alpha) : x\alpha=y\alpha \},$$
respectively. If $Y=[n]$, then $\alpha$ is called \emph{full transformation}. By $\mathcal{PT}_{n}$ and $\mathcal{T}_{n}$, we denote the set of all partial and full, respectively, transformations on $[n]$. Both sets, $\mathcal{PT}_{n}$ and $\mathcal{T}_{n}$, form monoids under the usual composition of mappings with the identity mapping $1_{n}$ as identity element. The empty transformation, i.e. its domain is the empty set, denoted by $0_{n}$, is a left- and right-zero in $\mathcal{PT}_{n}$. For any subset $Y$ of $[n]$, let $1_{Y}$ be the mapping $1_{n}$ restricted to $Y$. Such a mapping $1_{Y}$ is called \emph{partial identity}. Clearly, $1_{[n]}=1_{n}$ and $1_{\emptyset }=0_{n}$.

\smallskip

Let $<$ be the canonical linear order on $[n]$. A partial transformation $\alpha \in \mathcal{PT}_{n}$ is called \emph{order-preserving} if for all $x,y\in \dom(\alpha )$ with $x<y$, we have $x\alpha \leq y\alpha$. Moreover, $\alpha$ is called \emph{order-decreasing} if $x\alpha \leq x$ for all $x\in \dom(\alpha)$. The set of all order-preserving partial transformations on $[n]$ and the set of all order-decreasing partial transformations on $[n]$ are denoted by $\mathcal{PO}_{n}$ and $\mathcal{PD}_{n}$, respectively. Both sets, $\mathcal{PO}_{n}$ and $\mathcal{PD}_{n}$, form submonoids of $\mathcal{PT}_{n}$, which are already well studied (see, for example \cite{AAK,D3,F2,HZ,LU2,ZHQ}).

A linear order is a particular partial order and it appears naturally the question of the study of semigroups of transformations preserving another canonical partial order. A partial order "near" to the linear order is the so-called zig-zag order (fence). We consider a partial order $\prec$ on $[n]$ such that each element is either maximal or minimal, defined by 
\begin{eqnarray*}
	1 &\prec &2\succ 3\prec 4\succ \ldots \prec n\text{ if }n\text{ is even and} \\
	1 &\prec &2\succ 3\prec 4\succ \ldots \succ n\text{ if }n\text{ is odd }
\end{eqnarray*}%
or 
\begin{eqnarray*}
	1 &\succ &2\prec 3\succ 4\prec \ldots \prec n\text{ if }n\text{ is odd and} \\
	1 &\succ &2\prec 3\succ 4\prec \ldots \succ n\text{ if }n\text{ is even.}
\end{eqnarray*}%
In the first case, we call $\prec$ an \emph{up-fence} and in the second case a
\emph{down-fence}. As in the case of the linear order, an $\alpha \in \mathcal{PT}_{n}$ is called \emph{fence-preserving} if for all $x,y\in \dom(\alpha )$ with $x\prec y$, we have $x\alpha \preceq y\alpha$, and $\alpha $ is called \emph{fence-decreasing} if $x\alpha \preceq x$ for all $x\in \dom(\alpha )$. Since up-fences and down-fences are dually, it is enough to consider up-fences in this paper without to mention it. Since $[n]$ is a chain if $n\leq 2$, we suppose that $n\geq 3$ throughout this paper otherwise is stated.
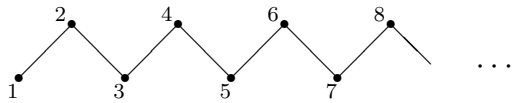
\begin{equation*}
	\begin{picture}(300,70)
		\put(30,40){\circle*{3}}
		\put(50,60){\circle*{3}}
		\put(70,40){\circle*{3}}
		\put(90,60){\circle*{3}}
		\put(110,40){\circle*{3}}
		\put(130,60){\circle*{3}}
		\put(150,40){\circle*{3}}
		\put(170,60){\circle*{3}}
		\put(50,60){\line(-1,-1){20}}
		\put(50,60){\line(1,-1){20}}
		\put(90,60){\line(-1,-1){20}}
		\put(90,60){\line(1,-1){20}}
		\put(130,60){\line(-1,-1){20}}
		\put(130,60){\line(1,-1){20}}
		\put(170,60){\line(-1,-1){20}}
		\put(170,60){\line(1,-1){15}}
		\put(28,33){\makebox(0,0){$^1$}}
		\put(46,64){\makebox(0,0){$_2$}}
		\put(68,33){\makebox(0,0){$^3$}}
		\put(86,64){\makebox(0,0){$_4$}}
		\put(108,33){\makebox(0,0){$^5$}}
		\put(126,64){\makebox(0,0){$_6$}}
		\put(148,33){\makebox(0,0){$^7$}}
		\put(166,64){\makebox(0,0){$_8$}}
		\put(210,44){\makebox(0,0){$\cdots $}}
	\end{picture}
	\vspace{-10mm}
\end{equation*}
If we consider a down-fence, then we will note it explicitly. Fences were firstly considered in relation to transformations by Rutkowski \cite{R} as well as Currie and Visentin \cite{CV}. Let $\mathcal{PF}_{n}$ ($\mathcal{PDF}_{n}$) be the set of all fence-preserving (fence-decreasing) partial transformations on $[n]$, and let $\mathcal{F}_{n} = \mathcal{PF}_{n} \cap \mathcal{T}_{n}$. All the sets $\mathcal{F}_{n}$, $\mathcal{PF}_{n}$, and $\mathcal{PDF}_{n}$ form submonoids of $\mathcal{PT}_{n}$. The monoids $\mathcal{F}_{n}$ and $\mathcal{PF}_{n}$ are already well studied (see, for example \cite{F1,J3,J2,PS,J1}). But $\mathcal{PDF}_{n}$ is not yet studied and we will investigate this monoid in the present paper. 

\smallskip

If we consider the linear order, then the monoid $\mathcal{PC}_{n}= \mathcal{PO}_{n} \cap \mathcal{PD}_{n}$ is a \emph{Catalan monoid}. This monoid is well studied (see, for example \cite{PH,LU3}). In the present paper, we will also study the fence counterpart of $\mathcal{PC}_{n}$, the monoid $\mathcal{PCF}_{n}=\mathcal{PF}_{n} \cap \mathcal{PDF}_{n}$. Moreover, we will consider the monoid $\mathcal{DF}_{n}= \mathcal{PDF}_{n} \cap \mathcal{T}_{n}$ of all full fence-decreasing transformations on $[n]$. It will turn out that $\mathcal{DF}_{n}$ forms a band. For $\mathcal{U}_{n}\in \{\mathcal{DF}_{n},\mathcal{PDF}_{n},\mathcal{PCF}_{n}\}$ and $0\leq r\leq n-1$, let 
$$\mathcal{U}(n,r)=\{\alpha \in \mathcal{U}_{n} : \left\vert \im(\alpha )\right\vert \leq r\}.$$
Clearly, $\mathcal{U}(n,r)$ is an ideal of $\mathcal{U}_{n}$.

\smallskip

Recall that the \emph{rank} of a finite semigroup $S$ is the minimal size of a generating set of $S$ and it is denoted by $\rank(S)$. An element $s\in S$ with $s^{2}=s$ is called \emph{idempotent}. The set of all idempotents of a subset $U$ of $S$ is denoted by $E(U)$. If there is a generating set $A$ of $S$ consisting entirely of idempotents, then we denote the minimal size of a generating set of $S$ consisting entirely of idempotents as the idempotent rank of $S$, in symbols: $\idrank(S)$, in particular, $S$ is called idempotent generated. Clearly, if $S$ has a unique minimal generating set $A$, then $rank(S)$ is the cardinality of $A$. An element $s\in S$ is called \emph{undecomposable} in $S$ if $s\neq s_{1}s_{2}$ for all $s_{1},s_{2}\in S\setminus \{s\}$. Clearly, every generating set of $S$ must contain all undecomposable elements of $S$. For other terms in semigroup theory, we refer to \cite{GM,H}.

\smallskip

In the present paper, we determine the cardinality of $\mathcal{DF}_{n}$, and the unique minimal generating sets, and so the ranks of $\mathcal{DF}_{n}$ and $\mathcal{DF}(n,r)$, $\left\lfloor \frac{n+1}{2} \right\rfloor \leq r\leq n-1$. For the monoid $\mathcal{PDF}_{n}$, we determine the cardinality, the number of its idempotents, and the unique minimal generating set. It will turn out that $\rank(\mathcal{PDF}_{n}) = \idrank(\mathcal{PDF}_{n})=2n$. Moreover, we determine the unique minimal generating sets for all ideals $\mathcal{PDF}(n,r)$, $0\leq r\leq n-1$. We obtain corresponding results for the monoid $\mathcal{PCF}_{n}$ and the ideals $\mathcal{PCF}(n,r)$, $0\leq r\leq n-1$.

\section{Fence-decreasing full transformations on $[n]$}

For a real number $x$, we denote the greatest integer not more than $x$ by $\lfloor x\rfloor$ and the least integer not less than $x$ by $\lceil x\rceil$. Let $\alpha\in \mathcal{DF}_{n}$. Since each element of $[n]^{o}$ is minimal, we have $i\alpha =i$ for every $i\in [n]^{o}$, and so $[n]^{o} \subseteq \fix(\alpha) \subseteq \im(\alpha)$. Thus, we notice that
$$ \left\lfloor \frac{n+1}{2} \right\rfloor \leq \lvert \im(\alpha) \rvert \quad \mbox{and}\quad \defect(\alpha) \subseteq [n]^{e}.$$
Moreover, if $i\in [n]^{e}\setminus \{n\}$, then we note that $i\alpha \in \{i-1,i,i+1\}$, and $n\alpha \in \{n-1,n\}$ if $n$ is even. Therefore, we immediately have
$$\lvert \mathcal{DF}_{n} \rvert =\left\{ 
\begin{array}{ll}
	3^{m} &\mbox{if $n$ is odd,}\\
	2\cdot 3^{m} &\mbox{if $n$ is even,}
\end{array}\right.$$
where $m=\left\lfloor \frac{n-1}{2} \right\rfloor$. For each $\left\lfloor \frac{n+1}{2} \right\rfloor \leq r\leq n$, let 
$$J_{r}=J_{n,r}=\{ \alpha \in \mathcal{DF}_{n} : \lvert \im(\alpha) \rvert=r \}.$$
Then, it is clear that $J_{n,n}=\{ 1_{n} \}$. More generally, we have the following result.

\smallskip

\begin{proposition} \label{p1} 
	Let $q=\left\lfloor \frac{n+1}{2} \right\rfloor$. Then, for each $q\leq r\leq n-1$, we have 
	$$\lvert J_{n,r} \rvert =\left\{ \begin{array}{ll}
		\binom{q-1}{r-q} 2^{n-r} & \mbox{if $n$ is odd,}\vspace{1mm} \\
		\frac{r}{q} \binom{q}{r-q} 2^{n-r-1} & \mbox{if $n$ is even.}
	\end{array}\right.$$
	In particular, $\lvert J_{n,n-1} \rvert =n-1$. 
\end{proposition}

\proof Let $\alpha \in \mathcal{DF}_{n}$ and let $q\leq r\leq n-1$. First, let $n$ be odd, i.e. $n=2q-1$ for some $q\in \mathbb{N}$. Since $[n]^{o} \subseteq \im(\alpha)$, the cardinality of $[n]^{o}$ is $q$, and since $i\alpha \in\{i-1 ,i,i+1\}$ for every $i\in [n]^{e}$, we conclude that $\alpha \in J_{n,r}$ if and only if there exist exactly $r-q$ elements in $[n]^{e}$ which are fixed. Thus, since the cardinality of $[n]^{e}$ is $q-1$, and since there exist $\binom{q-1}{r-q}$ subsets of $[n]^{e}$ of size $r-q$, it follows that 
$$\lvert J_{n,r} \rvert= \binom{q-1}{r-q} 2^{(q-1)-(r-q)} =\binom{q-1}{r-q} 2^{n-r}.$$

\smallskip

Let $n$ be even, i.e. $n=2q$ for some $q\in \mathbb{N}$. Recall that either $n\alpha =n-1$ or $n\alpha =n$. If we let 
$$J'_{n,r}=\{ \alpha \in J_{n,r} : n\alpha =n-1 \}\quad \mbox{ and }\quad J''_{n,r} =\{ \alpha \in J_{n,r} : n\alpha =n \},$$
then it is clear that $J_{n,r}$ is a disjoint union of $J'_{n,r}$ and $J''_{n,r}$, and that $J''_{n,r} =\emptyset$ if $r=\frac{n}{2}$. Since $\lvert J'_{n,r} \rvert= \lvert J_{n-1,r} \rvert$ and $\lvert J''_{n,r} \rvert= \lvert J_{n-1,r-1} \rvert$, it follows from the odd case that 
$$\lvert J_{n,r} \rvert= \binom{q-1}{r-q} 2^{n-1-r} + \binom{q-1}{r-1-q} 2^{n-r} = \frac{r}{q} \binom{q}{r-q} 2^{n-r-1}$$
(with the assumption that $\binom{\,\, p\,}{-1} =0$ for every positive integer $p$). 

\smallskip

Furthermore, it is clear that $\lvert J_{n,n-1} \rvert =n-1$ whatever the party of $n$ is. \qed

For each $j\in [n]^{e}$, we define two mappings $\lambda_{1j} : [n] \rightarrow [n]$ and  $\lambda_{2j} : [n] \rightarrow [n]$ (if $j<n$) by
\begin{eqnarray}\label{e1}
	i\lambda_{1j} =\left\{\begin{array}{cl}
		i  & \mbox{if } i\in [n] \setminus \{j\},\\
		j-1 & \mbox{if } i=j
	\end{array}	\right. &\mbox{and}&
	i\lambda_{2j} =\left\{\begin{array}{cl}
		i  & \mbox{if } i\in [n] \setminus \{j\},\\
		j+1 & \mbox{if } i=j.
	\end{array}	\right.
\end{eqnarray}
With these notations, it is easy to verify that 
$$J_{n, n-1}=\{ \lambda_{1j} : j\in [n]^{e}\} \cup \{ \lambda_{2j} : j\in [n]^{e} \setminus \{n\}\}.$$

\begin{lemma} \label{l2} 
	$\mathcal{DF}_{n}$ is a band, i.e., each element of $\mathcal{DF}_{n}$ is an idempotent.
\end{lemma}

\proof We already noticed that $[n]^{o} \subseteq \fix(\alpha)$ for all  $\alpha \in \mathcal{DF}_{n}$. Let $\alpha \in \mathcal{DF}_{n}$ and let $i\in [n]^{e}$. Since either $i\alpha =i$ or $i\alpha$ is odd, we have $i\alpha \in \fix(\alpha)$. Thus, $i\alpha^{2}=i\alpha$, and so $\alpha$ is an idempotent. \qed

By Lemma \ref{l2}, we conclude that $\mathcal{DF}_{n}$ is a regular semigroup, i.e. for each element $\alpha \in \mathcal{DF}_{n}$, there exists $\beta \in \mathcal{DF}_{n}$ such that $\alpha \beta \alpha =\alpha$, namely $\beta= \alpha$. Moreover, it is clear that $\mathcal{DF}_{n}$ is not commutative, and that $\alpha \in \mathcal{DF}_{n}$ is a left zero if and only if $\im(\alpha) =[n]^{o}$.

\smallskip

If we suppose $[n]$ is a down-fence, then one can notice that
$$[n]^{e} \subseteq \fix(\alpha),\quad \left\lceil \frac{n-1}{2} \right\rceil \leq \lvert \im(\alpha) \rvert \quad \mbox{and}\quad \defect(\alpha) \subseteq [n]^{o}$$ 
for each $\alpha\in \mathcal{DF}_{n}$. Furthermore, if we let $m=\left\lfloor \frac{n-1}{2} \right\rfloor$, then we immediately have 
$$\lvert \mathcal{DF}_{n} \rvert =\left\{ 
\begin{array}{cl}
	4\cdot 3^{m-1} &\mbox{if $n$ is odd,}\\
	2\cdot 3^{m-1} &\mbox{if $n$ is even.}
\end{array}\right.$$
Thus, $\mathcal{DF}_{n}$ changes if we reverse the direction of the fence.

\smallskip

\begin{proposition} \label{p3} 
	Let $q=\left\lfloor \frac{n+1}{2} \right\rfloor$. Then $J_{r} \subseteq \langle J_{r+1} \rangle$ for each $q\leq r\leq n-2$, and so $\mathcal{DF}(n,r) =\langle J_{r}\rangle$ for each $q\leq r\leq n-1$. In particular, $\mathcal{DF}_{n} =\langle J_{n-1}\cup \{1_{n}\}\rangle$.
\end{proposition}

\proof Let $q\leq r\leq n-2$ and let $\alpha \in J_{r}$. Then there exist two distinct elements $j$ and $k$ in $[n]^{e}$ such that $j,k\in \defect(\alpha)$. If we consider the mappings $\beta : [n]\rightarrow [n]$ and $\gamma :[n]\rightarrow [n]$ defined by
\begin{eqnarray*}
	i\beta =\left\{\begin{array}{cl}
		i\alpha &\mbox{if } i\in [n]\setminus \{j\},\\
		j       &\mbox{if } i=j
	\end{array}	\right. &\mbox{and}&
	i\gamma =\left\{\begin{array}{cl}
		i\alpha &\mbox{if } i\in [n]\setminus \{k\},\\
		k       &\mbox{if } i=k,
	\end{array}	\right.
\end{eqnarray*} 
then it is clear that $\beta, \gamma \in J_{r+1}$ and $j\beta \gamma =j\gamma =j\alpha$. Moreover, it follows from Lemma \ref{l2} that $i(\beta \gamma) =(i\alpha) \gamma= i\alpha^{2} =i\alpha$ for every $i\in [n]\setminus \{j\}$. Therefore, we have the equality $\alpha =\beta \gamma$, as required. \qed

Using the notations employed in the proof above, we have the equality $\alpha =\gamma \beta$, as well.

\smallskip

As on chains, one can easily show that $\fix(\beta \gamma )= \fix(\beta) \cap \fix(\gamma)$ for all $\beta, \gamma \in \mathcal{DF}_{n}$ (see, for example \cite[Lemma 1.1]{LU}). Since $\mathcal{DF}_{n}$ is a band, we have $\im(\delta) =\fix(\delta)$ for all $\delta \in \mathcal{DF}_{n}$, and thus,
\begin{eqnarray*}
\im(\beta \gamma) =\im(\beta) \cap \im(\gamma),\, \mbox{ and so }\, \im(\beta \gamma) \subseteq \im(\beta)
\end{eqnarray*}
for all $\beta, \gamma \in \mathcal{DF}_{n}$.

\smallskip

\begin{proposition}\label{p4}
	Each transformation in $\mathcal{DF}_{n}$ is uniquely determined by the kernel and the image. 
\end{proposition}

\proof Assume that there are $\alpha, \beta\in \mathcal{DF}_{n}$ with $\ker(\alpha) =\ker(\beta)$ and $\im(\alpha) =\im(\beta)$ such that $\alpha \neq \beta$. Then there is a smallest integer $i$ in $[n]$ such that $i\alpha \neq i\beta$. It is clear that $i$ is even, and so $i\alpha, i\beta \in \{i-1,i,i+1\}$ if $i\neq n$ and $n\alpha, n\beta \in \{n-1,n\}$ if $i=n$. Without loss of generality, we can assume that $i\alpha \in \{i-1,i\}$. Suppose that $i\alpha =i$. Since $\im(\alpha) =\im(\beta)$, there is $j\in [n]$ such that $j\beta=i$, i.e. $j=i$ since $i$ is even. Thus $i=j\beta =i\beta\neq i\alpha=i$, which is a contradiction. Suppose that $i\alpha=i-1$. Since $\im(\alpha) =\im(\beta)$, there is $k\in \dom(\beta)$ such that $k\beta =i-1$. Since $k$ must be in $\{ i-2,i-1,i\}$, it follows from the choice of $i$ that  $k\alpha =k\beta$, and so we conclude that $k\alpha =k\beta =i-1= i\alpha$. Thus, $(i,k)\in \ker(\alpha) =\ker(\beta)$, and so $i-1\neq i\beta =k\beta =i-1$, which is a contradiction. \qed

\begin{lemma} \label{l5} 	
	Let $\left\lfloor \frac{n+1}{2} \right\rfloor \leq r\leq n-1$. Then each element of $J_{r}$ is undecomposable in $\mathcal{DF}(n,r)$. In particular, each element of $J_{n-1}\cup \{1_{n}\}$ is undecomposable in $\mathcal{DF}_{n}$.
\end{lemma}

\proof Let $\alpha \in J_{r}$, and suppose that $\alpha =\beta \gamma $ for some $\beta ,\gamma \in \mathcal{DF}(n,r)$. Since $\ker(\beta) \subseteq \ker(\alpha)$,\, $\im(\alpha) \subseteq \im(\beta)$ and $\lvert \im(\beta) \rvert\leq r=\lvert \im(\alpha) \rvert$, it follows that $\im(\alpha) =\im(\beta)$ and $\ker(\alpha) =\ker(\beta)$. Thus, by Proposition \ref{p4},  we have $\alpha =\beta$, as required. 

\smallskip

The fact that $1_{n}$ is the unique element in $J_{n}$ completes the proof.\qed

\begin{theorem} \label{t6} 
	Let $q=\left\lfloor \frac{n+1}{2} \right\rfloor$. Then $J_{r}$ is the unique minimal generating set of  $\mathcal{DF}(n,r)$ for each $q\leq r\leq n-1$, and $J_{n-1}\cup \{1_{n}\}$ is the unique minimal generating set of  $\mathcal{DF}_{n}$. Therefore, for each $q\leq r\leq n-1$,
	$$\rank(\mathcal{DF}(n,r)) =\idrank( \mathcal{DF}(n,r)) =\left\{ \begin{array}{ll}
		\binom{q-1}{r-q} 2^{n-r} & \mbox{if $n$ is odd,}\vspace{1mm} \\
		\frac{r}{q} \binom{q}{r-q} 2^{n-r-1} & \mbox{if $n$ is even,}
	\end{array}\right.$$
	and $\rank(\mathcal{DF}_{n}) =\idrank( \mathcal{DF}_{n}) =n$.
\end{theorem}

\proof The results follow from Propositions \ref{p1} and \ref{p3}, and Lemmas \ref{l2} and \ref{l5}. \qed

\section{Fence-decreasing partial transformations on $[n]$}

Let $\alpha\in \mathcal{PDF}_{n}$. For each $i\in [n]^{o}$, we have either $i\notin \dom(\alpha)$ or $i\alpha =i$, and so $[n]^{o} \cap \dom(\alpha) \subseteq \fix(\alpha)$. If $i\in [n]^{e}\setminus \{n\}$, then either $i\notin \dom(\alpha)$ or $i\alpha \in \{i-1,i,i+1\}$, and if $n$ is even, then either $n\notin \dom(\alpha)$ or $n\alpha \in \{n-1,n\}$. Thus, with a simple calculation, we have  
$$\lvert \mathcal{PDF}_{n} \rvert =\left\{ 
\begin{array}{cl}
	2^{3m+1} &\mbox{if $n$ is odd,}\\
	3\cdot 2^{3m+1} &\mbox{if $n$ is even,}
\end{array} \right.$$
where $m=\left\lfloor \frac{n-1}{2} \right\rfloor$. For each $0\leq r\leq n$, let 
$$K_{r}=K_{n,r}=\{ \alpha \in \mathcal{PDF}_{n} : \lvert \im(\alpha) \rvert=r \}.$$
It is clear that $K_{n,n}=J_{n,n}=\{ 1_{n} \}$ and $K_{n,0}=\{ 0_{n} \}$. For $1\leq r\leq n-1$, let $a_{n,r}=\lvert K_{n,r} \rvert$. Since $a_{n,0}= a_{n,n}=1$, we determine the other cardinalities in the following lemma.

\smallskip

\begin{proposition} \label{p8} 
	Let $1\leq r\leq n-1$. Then we have 
	$$a_{n,r} =\left\{ \begin{array}{ll}
		a_{n-1,r}+ a_{n-1,r-1}+ 2a_{n-2,r-1} & \mbox{if $n$ is odd,}\vspace{1mm} \\
		2a_{n-1,r} +a_{n-1,r-1} +a_{n-2,r-1} -a_{n-2,r} & \mbox{if $n$ is even.}
	\end{array}\right.$$
\end{proposition}

\proof Suppose $n$ is odd. Since $n\alpha=n$ whenever $n\in \dom(\alpha)$ and $(n-2)\alpha \neq n-1$, it follows that $K_{n,r}$ is a disjoint union of the following sets:
$$\begin{array}{lll}
	K^{1}_{n,r} =\{ \alpha \in K_{n,r} : n\alpha^{-1} =\emptyset \},&&
	K^{2}_{n,r} =\{ \alpha \in K_{n,r} : n\alpha^{-1} =\{ n-1\} \}, \vspace*{1mm}\\
	K^{3}_{n,r} =\{ \alpha \in K_{n,r} : n\alpha^{-1} =\{ n\}\}, &\mbox{and}& K^{4}_{n,r} =\{ \alpha \in K_{n,r} : n\alpha^{-1} =\{ n-1,n\}\}.
\end{array}$$
Since $\lvert K^{1}_{n,r} \rvert= \lvert K_{n-1,r} \rvert$,\, $\lvert K^{2}_{n,r} \rvert= \lvert K^{4}_{n,r} \rvert= \lvert K_{n-2,r-1} \rvert$ and $\lvert K^{3}_{n,r} \rvert= \lvert K_{n-1,r-1} \rvert$, it follows that $a_{n,r}= a_{n-1,r}+ a_{n-1,r-1}+ 2a_{n-2,r-1}$.

\smallskip 

Suppose $n$ is even. If $n\in \dom(\alpha)$, then either $n\alpha =n$ or $n\alpha =n-1$. Since $(n-1)\alpha =n-1$ whenever $n-1\in \dom(\alpha)$, it follows that $K_{n,r}$ is a disjoint union of the following sets:
$$\begin{array}{l}
	K^{5}_{n,r} =\{ \alpha \in K_{n,r} : n\notin \dom(\alpha) \}, \qquad 
	K^{6}_{n,r} =\{ \alpha \in K_{n,r} : n\alpha^{-1} =\{n\} \}, \vspace*{1mm}\\
	K^{7}_{n,r} =\{ \alpha \in K_{n,r} : (n-1)\alpha^{-1} =\{n\} \},\, \mbox{ and} \vspace*{1mm}\\
	K^{8}_{n,r} =\{ \alpha \in K_{n,r} : n\in (n-1)\alpha^{-1} \mbox{ and } (n-1)\alpha^{-1} \neq \{n\} \}.
\end{array}$$
It is clear that $\lvert K^{5}_{n,r} \rvert= \lvert K_{n-1,r} \rvert$,\, $\lvert K^{6}_{n,r} \rvert= \lvert K_{n-1,r-1} \rvert$  and  $\lvert K^{7}_{n,r} \rvert= \lvert K_{n-2,r-1} \rvert$. Moreover, since $\alpha \in K^{8}_{n,r}$ if and only if $\alpha_{\mid_{[n-1]}} \in K_{n-1,r}$ with $(n-2)\alpha =n-1$ or $(n-1)\alpha =n-1$, it follows that $\lvert K^{8}_{n,r} \rvert= \lvert K_{n-1,r} \rvert -\lvert K_{n-2,r} \rvert$, and so $a_{n,r} = 2a_{n-1,r} +a_{n-1,r-1} +a_{n-2,r-1} -a_{n-2,r}$, as claimed. \qed

\begin{table}[!ht]	
	\begin{center} \label{Tab:card1}
		\begin{tabular}{c|cccccccc|c}	
			n$\backslash$r  & 0 & 1 & 2 & 3  & 4  & 5 & 6 & 7 & $\lvert\mathcal{PDF}_{n}\rvert$ \\	\hline \vspace*{-5mm}\\
			1  & 1 &   1 &     &     &    &    &   &   &    2  \\
			\hline \vspace*{-5mm}\\
			2  & 1 &   4 &   1 &     &    &    &   &   &    6  \\
			\hline \vspace*{-5mm}\\
			3  & 1 &   7 &   7 &   1 &    &    &   &   &   16  \\
			\hline \vspace*{-5mm}\\
			4  & 1 &  12 &  24 &  10 &  1 &    &   &   &   48  \\
			\hline \vspace*{-5mm}\\
			5  & 1 &  15 &  50 &  48 & 13 &  1 &   &   &   128  \\
			\hline \vspace*{-5mm}\\
			6  & 1 & 20 & 103 & 160 & 83 & 16  & 1 &    &  384  \\
			\hline \vspace*{-5mm}\\
			7  & 1 & 23 & 153 & 363 & 339 & 125 & 19 & 1 & 1024  \\
		\end{tabular}
		\caption{$\lvert K_{n,r} \rvert$ and $\lvert \mathcal{PDF}_{n}\rvert$.}
	\end{center}
\end{table}

\smallskip

\begin{lemma} \label{l9} 
	Each $\alpha \in \mathcal{PDF}_{n}$ is either an idempotent or $\alpha^{3} =\alpha^{2}$, i.e. $\alpha$ is a quasi-idempotent. Moreover, $\alpha\in \mathcal{PDF}_{n}$ is regular if and only if  $\alpha\in E(\mathcal{PDF}_{n})$.
\end{lemma}

\proof For every $\alpha \in \mathcal{PDF}_{n} \setminus E(\mathcal{PDF}_{n})$, it is enough to show that $\alpha^{3} =\alpha^{2}$. Let $\alpha \in \mathcal{PDF}_{n} \setminus E(\mathcal{PDF}_{n})$, and let $i\in \dom(\alpha^{2}) \subseteq \dom(\alpha)$. If $i\in \fix(\alpha)= \fix(\alpha^{2})$, then it is clear that $i\alpha^{3} =i\alpha^{2}$. Suppose $i\notin \fix(\alpha)$. Since $i$ must be even and $i\alpha \in [n]^{o} \cap \dom(\alpha)$, we obtain that $i\alpha^{2} =i\alpha$, and so $i\alpha^{3} =i\alpha^{2}$. Thus, $\alpha^{3} =\alpha^{2}$.

\smallskip

Moreover, since idempotents are regular, we consider the non-idempotents elements of $\mathcal{PDF}_{n}$. Let $\alpha\in \mathcal{PDF}_{n} \setminus E(\mathcal{PDF}_{n})$. Then there exists $j\in \dom(\alpha) \cap [n]^{e}$ such that either  $j\alpha =j-1$ and $j-1 \notin \dom(\alpha)$, or $j\alpha =j+1$ and $j+1 \notin \dom(\alpha)$. If we assume $\alpha =\alpha \beta \alpha$ for some  $\beta\in \mathcal{PDF}_{n}$, then we immediately obtain the contradiction $j\notin \dom(\alpha \beta \alpha)$ in both cases. \qed

If we suppose that $[n]$ is a chain, then it is known that $\mathcal{PD}_{n}$ and $\mathcal{D}_{n+1}$ are isomorphic. However, by Lemmas \ref{l2} and \ref{l9}, we can conclude that $\mathcal{PDF}_{n}$ and $\mathcal{DF}_{n+1}$ are not isomorphic.

\smallskip

\begin{proposition} \label{p10} 
	For each $0 \leq r\leq n-2$, we have $K_{r} \subseteq \langle K_{r+1} \rangle$, and so for each $1\leq r\leq n-1$, $\mathcal{PDF}(n,r) =\langle K_{r}\rangle$. Thus, $\mathcal{PDF}_{n} =\langle K_{n-1}\cup \{1_{n}\}\rangle$.
\end{proposition}

\proof Let $0 \leq r\leq n-2$ and let $\alpha \in K_{r}$. Suppose $\defect(\alpha) \cap [n]^{o}\neq \emptyset$. Then we let $j\in \defect(\alpha) \cap [n]^{o}$, and so $j\notin \dom(\alpha)$, and let $k\in \defect(\alpha) \setminus \{j\}$. If we define the mappings $\beta : \dom(\alpha)\cup \{j\} \rightarrow [n]$ and $\gamma : \im(\alpha)\cup \{k\} \rightarrow [n]$ by
\begin{eqnarray*}
	i\beta =\left\{\begin{array}{cl}
		i\alpha &\mbox{if } i\in \dom(\alpha),\\
		j       &\mbox{if } i=j
	\end{array}	\right. &\mbox{ and }&
	i\gamma =\left\{\begin{array}{cl}
		i &\mbox{if } i\in \im(\alpha),\\
		k       &\mbox{if } i=k,
	\end{array}	\right.
\end{eqnarray*} 
then it clear that $\beta, \gamma \in K_{r+1}$ and $\beta \gamma =\alpha$. 

\smallskip

Suppose $\defect(\alpha) \cap [n]^{o}= \emptyset$, i.e. $[n]^{o} \subseteq \im(\alpha)$ and  $\defect(\alpha) \subseteq [n]^{e}$. Let $j$ and $k$ be two distinct elements of $\defect(\alpha)$ with $j<k$, and let 
\begin{eqnarray*}
	U=\dom(\alpha) \cup \{j\} &\mbox{and}&
	V= \left\{\begin{array}{ll}
		\im(\alpha)\cup \{j,k\} &\mbox{if } j\in \dom(\alpha),\\
		\im(\alpha)\cup \{k\} &\mbox{if } j\notin \dom(\alpha).
	\end{array}	\right.
\end{eqnarray*}
Then we define the mappings $\beta : U \rightarrow [n]$ and $\gamma : V \rightarrow [n]$ by
\begin{eqnarray*}
	i\beta =\left\{\begin{array}{cl}
		i\alpha &\mbox{if } i\in U\setminus \{j\},\\
		j       &\mbox{if } i=j
	\end{array}	\right. &\mbox{and}&
	i\gamma =\left\{\begin{array}{cl}
		i       &\mbox{if } i\in \im(\alpha),\\
		j\alpha &\mbox{if } i=j \mbox{ and } j\in \dom(\alpha),\\
		k       &\mbox{if } i=k.
	\end{array}	\right.
\end{eqnarray*} 
First, it is clear that $\gamma \in K_{r+1}$ and $\beta \gamma =\alpha$. Moreover, if $j\notin \dom(\alpha)$, or if  $j, j\alpha \in \dom(\alpha)$, then it is also clear that  $\beta \in K_{r+1}$. If  $j\in \dom(\alpha)$ and $j\alpha \notin \dom(\alpha)$, then $j\alpha \in \defect(\beta) \cap [n]^{o}$, and so by the first part of the proof, $\beta \in \langle K_{r+1}\rangle$, which completes the proof. \qed

We omit the proof of the following proposition since the proof is almost the same with the proof of Proposition \ref{p4}. 

\smallskip

\begin{proposition}\label{p11}
	Each transformation in $\mathcal{PDF}_{n}$ is uniquely determined by the kernel and the image. \hfill $\square$
\end{proposition}

Notice that for $1\leq r\leq n-1$, not all elements of $K_{r}$ are undecomposable in $\mathcal{PDF}(n,r)$. For  example, we have $\alpha= \left(\begin{matrix}
	1&2&3&4\\
	1&1&3&5
\end{matrix}\right) = \left(\begin{matrix}
	1&2&3&4\\
	1&1&3&4
\end{matrix}\right) \left(\begin{matrix}
	1&3&4\\
	1&3&5
\end{matrix}\right)$ in $\mathcal{PDF}(5,3)$. However, it is easy to see that 
$\left(\begin{matrix}
2&4&6&8\\
3&3&7&7
\end{matrix}\right)$ 
is undecomposable in $\mathcal{PDF}(8,2)$.  

\smallskip 

For $0\leq r\leq n$, let
$$K^{*}_{r} =K^{*}_{n,r} =\{ \alpha \in K_{r} : \sh(\alpha) \neq \emptyset \mbox{ and } j\alpha^{-1} =\{j-1,j+1\} \mbox{ for all } j\in \sh(\alpha) \}.$$
Note that $j\in \sh(\alpha)$ implies $j\notin \dom(\alpha)$ and $j\in [n]^{o} \setminus \{1,n\}$ for any $\alpha \in K_{r}^{*}$. It is clear that $K^{*}_{n,r} =\emptyset$ for every $0\leq r\leq n\leq 3$, and that $K^{*}_{0} =K^{*}_{n-1} =K^{*}_{n} =\emptyset$ for every $n\geq 4$.

\smallskip 

\begin{lemma} \label{l12} 
	\begin{itemize}
		\item [$(i)$] Let $1\leq r\leq n$. Then each element of $E(K_{r})$ is undecomposable in $\mathcal{PDF}(n,r)$. 
		\item [$(ii)$] Let $n\geq 4$ and $1\leq r\leq n-2$. Then each element of $K^{*}_{r}$ is undecomposable in $\mathcal{PDF}(n,r)$. 
	\end{itemize}
\end{lemma}

\proof We omit the proof of $(i)$ since the proof is almost the same with the proof of Lemma \ref{l5}.

Let $\alpha \in K^{*}_{r}$, and let $j\in \sh(\alpha)$, i.e. $(j-1)\alpha =j= (j+1)\alpha$ and $j\notin \dom(\alpha)$. Suppose that $\alpha =\beta \gamma$ for some $\beta, \gamma \in \mathcal{PDF}(n,r)$. Then it is clear that $\fix(\alpha) \subseteq \fix(\gamma)$,\, $\fix(\alpha) \subseteq \im(\beta)$, and that $j-1,j+1\in \dom(\beta) \cap [n]^{e}$. Since $(j-1)\beta =(j-2)$ yields the contradiction $j=(j-1) \alpha =(j-1)\beta \gamma =(j-2)\gamma \neq j$, we have $(j-1)\beta \in \{j-1,j\}$. Assume that $(j-1)\beta =j-1$. Since $j-1\notin \im(\alpha)$, but $j-1\in \im(\beta)$, to avoid changing the size of the image set, we must have $(j+1)\beta =j+2$ and $j+2\in \im(\alpha)$, which gives the contradiction $j =(j+1)\alpha =(j+1) \beta \gamma =(j+2) \gamma \neq j$. Thus, we conclude that $(j-1)\beta =j= (j-1)\alpha$.   

\smallskip

One can similarly show that $(j+1)\beta =j= (j+1)\alpha$. Therefore, we have $\im(\beta) = \im(\alpha) \subseteq \fix(\gamma)$ and $\beta_{\mid_{\dom(\alpha)}} =\alpha$. Since $i\beta =(i\beta) \gamma =i\alpha$ for each $i\in \dom(\beta)$, it follows that $\dom(\beta)=\dom(\alpha)$, and so $\beta= \alpha$, as required. \qed

Notice that if $\left\lfloor \frac{n+1}{2} \right\rfloor \leq r\leq n-1$, then it is clear that $J_{r} \subseteq E(K_{r})$. Let $\alpha \in\mathcal{PDF}_{n}$, and let
\begin{eqnarray*}
	&& A_{\alpha} =\{ j\in \dom(\alpha) \cap [n]^{e} : j\alpha \notin \dom(\alpha)\} \mbox{ and}\\
	&& A_{\alpha}^{*} =\{ j\in A_{\alpha} : \lvert (j\alpha) \alpha^{-1} \rvert=1\}.
\end{eqnarray*}
Then it is clear that $\alpha \in \mathcal{PDF}_{n}$ is idempotent if and only if $A_{\alpha} =\emptyset$. 

\smallskip

\begin{proposition} \label{p13} 
	Let $1\leq r\leq n-2$ and let $\alpha\in K_{r}$. Then $\alpha \in \langle E(K_{r}) \cup K^{*}_{r}\rangle$.
\end{proposition}

\proof  Let $\alpha\in K_{r}\setminus (E(K_{r}) \cup K^{*}_{r})$. Then it is clear that $A_{\alpha}$ is not empty. For each $j\in A_{\alpha}$, we first notice that $j\alpha \in [n]^{o}$ and that if $j\alpha=j-1$, then $(j\alpha) \alpha^{-1} \subseteq \{ j-2,j\}$, and if $j\alpha=j+1$, then $(j\alpha) \alpha^{-1} \subseteq \{ j,j+2\}$. Thus, $1\leq \lvert (j\alpha) \alpha^{-1} \rvert \leq 2$ for each $j\in A_{\alpha}$. Now it is clear that $\alpha\in K^{*}_{r}$ if and only if $\lvert (j\alpha) \alpha^{-1} \rvert=2$ for all $j\in A_{\alpha}$, or equivalently $A_{\alpha}^{*} =\emptyset$.

\smallskip

Now assume that $A_{\alpha}^{*}$ is not empty. Then there is $j\in A_{\alpha}^{*}$ such that $j\alpha \in [n]^{o}$, and $i\alpha \neq j\alpha$ for all $i\in \dom(\alpha) \setminus \{j\}$. Thus, if we define the mappings $\beta : \dom(\alpha) \rightarrow [n]$ and $\gamma : \im(\alpha) \cup \{j\} \rightarrow [n]$ by
\begin{eqnarray*}
	i\beta =\left\{\begin{array}{cl}
		i\alpha &\mbox{if } i\in \dom(\alpha)\setminus \{j\},\\
		j       &\mbox{if } i=j
	\end{array}	\right. &\mbox{and}&
	i\gamma =\left\{\begin{array}{cl}
		i		&\mbox{if } i\in \im(\alpha),\\
		j\alpha &\mbox{if } i=j,
	\end{array}	\right.
\end{eqnarray*} 
then it is clear that $\beta\in K_{r}$,\, $\gamma \in E(K_{r})$ and $\alpha =\beta \gamma$. Thus, since $\lvert A_{\beta}^{*} \rvert= \lvert A_{\alpha}^{*}  \rvert -1 $, the proof is now completed by induction on $\lvert A_{\alpha}^{*}  \rvert$. \qed

For each $0\leq r\leq n$, let
$$E_{r}=\{ 1_{Y} : Y \mbox{ is a subset of } [n] \mbox{ with size }r\}.$$
It is clear that $E_{0}=\{ 0_{n}\}$,\, $E_{n}=\{ 1_{n}\}$,\, $E_{r} \subseteq E(K_{r})$, and that $\lvert E_{r} \rvert =\binom{n}{r}$ for each $0\leq r\leq n$. 

\smallskip

Let $e_{n}$ denote the number of idempotents in $\mathcal{PDF}_{n}$. Moreover, for $0\leq r\leq n$, let $e_{n,r}$ denote the number of idempotents in $K_{n,r}$, i.e. $e_{n,r}= \lvert E(K_{n,r}) \rvert$. Since $e_{n,0}=e_{n,n}=1$, we determine the other values of $e_{n,r}$ in the following proposition and in Table 2.

\smallskip

\begin{proposition} \label{p14} 
	Let $1\leq r\leq n-2$. Then we have 
	$$e_{n,r} =\left\{ \begin{array}{ll}
		e_{n-1,r}+ e_{n-1,r-1} +e_{n-2,r-1} & \mbox{if $n$ is odd,}\vspace{1mm} \\
		e_{n-1,r}+ e_{n-1,r-1} +e_{n-2,r-1} +e_{n-3,r-1} & \mbox{if $n$ is even.}
	\end{array}\right.$$
Moreover, we have $e_{n,n-1} =2n-1$ whatever the party of $n$ is, and 
$$e_{n} =\left\{ \begin{array}{ll}
	2e_{n-1}+ e_{n-2} & \mbox{if $n$ is odd,}\vspace{1mm} \\
	2e_{n-1}+ e_{n-2} +e_{n-3} & \mbox{if $n$ is even.}
\end{array}\right.$$
\end{proposition}

\proof By Lemma \ref{l2}, we notice that $J_{n-1} \subseteq E(K_{n-1})$. Since an element in $E(K_{n-1}) \setminus J_{n-1}$ must be a partial identity, it follows from Proposition \ref{p1} that $e_{n,n-1} =2n-1$.

\smallskip

Let $1\leq r\leq n-2$ and suppose $n$ is odd. Since $n\alpha=n$ whenever $n\in \dom(\alpha)$, it follows that $E(K_{n,r})$ is a disjoint union of the following sets:
\begin{eqnarray*}
	&&E^{1}_{n,r} =\{ \alpha \in E(K_{n,r}) : n\notin \dom(\alpha) \},\\
	&&E^{2}_{n,r} =\{ \alpha \in E(K_{n,r}) : n\alpha^{-1} =\{n\}\,\}, \mbox{ and}\\
	&&E^{3}_{n,r} =\{ \alpha \in E(K_{n,r}) : n\alpha^{-1} =\{n-1,n\}\,\}. 
\end{eqnarray*}
Since it can be easily seen that $\lvert E^{1}_{n,r} \rvert= \lvert E(K_{n-1,r}) \rvert$,\,  $\lvert E^{2}_{n,r} \rvert= \lvert E(K_{n-1,r-1})\rvert$, and $\lvert E^{3}_{n,r} \rvert= \lvert E(K_{n-2,r-1}) \rvert$, 
we conclude that $e_{n,r} =e_{n-1,r}+ e_{n-1,r-1} +e_{n-2,r-1}$. 

\smallskip

Suppose $n$ is even. Since $n\alpha=n$ or $n\alpha=n-1$  whenever $n\in \dom(\alpha)$, it follows that $E(K_{n,r})$ is a disjoint union of the following sets:
\begin{eqnarray*}
	&&E^{4}_{n,r} =\{ \alpha \in E(K_{n,r}) : n\notin \dom(\alpha)\},\\
	&&E^{5}_{n,r} =\{ \alpha \in E(K_{n,r}) : n\alpha^{-1} =\{n\} \},\\
	&&E^{6}_{n,r} =\{ \alpha \in E(K_{n,r}) : (n-1)\alpha^{-1} =\{n-1,n\}\}, \mbox{ and}\\
	&&E^{7}_{n,r} =\{ \alpha \in E(K_{n,r}) : (n-1)\alpha^{-1} =\{n-2,n-1,n\}\}. 
\end{eqnarray*}
Since it can be easily seen that $\lvert E^{4}_{n,r} \rvert= \lvert E(K_{n-1,r}) \rvert$,\,  $\lvert E^{5}_{n,r} \rvert= \lvert E(K_{n-1,r-1}) \rvert$,\, $\lvert E^{6}_{n,r} \rvert=  \lvert E(K_{n-2,r-1}) \rvert$, and $\lvert E^{7}_{n,r} \rvert= \lvert E(K_{n-3,r-1}) \rvert$, we conclude that $e_{n,r} =e_{n-1,r}+ e_{n-1,r-1} +e_{n-2,r-1} +e_{n-3,r-1}$. 

\smallskip

Moreover, if $n$ is odd, then $e_{n}= 2e_{n-1}+ e_{n-2}$ since $E(\mathcal{PDF}_{n})$ is a disjoint union of the following sets:
$$\begin{array}{ll}
	\{ \alpha \in E(\mathcal{PDF}_{n}) : n\notin \dom(\alpha) \}, &
	\{ \alpha \in E(\mathcal{PDF}_{n}) : n\alpha^{-1} =\{n\}\},\, \mbox{ and}\vspace{1mm}\\
	\{ \alpha \in E(\mathcal{PDF}_{n}) : n\alpha^{-1} =\{n-1,n\}\}, &
\end{array}$$
and if $n$ is even, then $e_{n}= 2e_{n-1}+ e_{n-2}+ e_{n-3}$ since $E(\mathcal{PDF}_{n})$ is a disjoint union of the following sets:
$$\begin{array}{l}
	\{ \alpha \in E(\mathcal{PDF}_{n}) : n\notin \dom(\alpha) \}, \qquad
	\{ \alpha \in E(\mathcal{PDF}_{n}) : n\alpha^{-1} =\{n\}\,\},\\
	\{ \alpha \in E(\mathcal{PDF}_{n}) : (n-1)\alpha^{-1} =\{n-1,n\}\},\, \mbox{ and}\\
	\{ \alpha \in E(\mathcal{PDF}_{n}) : (n-1)\alpha^{-1} =\{n-2, n-1,n\}\}, 
\end{array}$$
as claimed. \qed

Especially, as can be easily calculated, we have $e_{n,1} =\left\{ \begin{array}{ll}
	5m & \mbox{if $n$ is odd,}\vspace{1mm} \\
	5m-2 & \mbox{if $n$ is even,}
\end{array}\right.$
where $m= \lfloor \frac{n}{2}\rfloor$.

\smallskip

\begin{table}[!ht]	
	\begin{center} \label{Tab:card2}
		\begin{tabular}{c|cccccccc|c}	
			n$\backslash$r  & 0 & 1 & 2 & 3  & 4  & 5 & 6 & 7 & $\lvert E(\mathcal{PDF}_{n})\rvert$ \\	\hline \vspace*{-5mm}\\
			1  & 1 &   1 &     &     &    &    &   &   &    2  \\
			\hline \vspace*{-5mm}\\
			2  & 1 &   3 &   1 &     &    &    &   &   &    5  \\
			\hline \vspace*{-5mm}\\
			3  & 1 &   5 &   5 &   1 &    &    &   &   &   12  \\
			\hline \vspace*{-5mm}\\
			4  & 1 &  8 &  14 &  7 &  1 &    &   &   &   31  \\
			\hline \vspace*{-5mm}\\
			5  & 1 &  10 &  27 &  26 & 9 &  1 &   &   &   74  \\
			\hline \vspace*{-5mm}\\
			6  & 1 & 13 & 50 & 72 & 43 & 11  & 1 &    &  191  \\
			\hline \vspace*{-5mm}\\
			7  & 1 & 15 & 73 & 149 & 141 & 63 & 13 & 1 & 456 \\
		\end{tabular}
		\caption{$\lvert E(K_{n,r}) \rvert$ and $\lvert E(\mathcal{PDF}_{n})\rvert$.}
	\end{center}
\end{table}

\begin{theorem} \label{t15} 
	Let $1\leq r\leq n-1$. Then $E(K_{r}) \cup K^{*}_{r}$ is the unique minimal generating set of $\mathcal{PDF}(n,r)$. In particular, $E_{n-1}\cup J_{n-1}\cup \{1_{n}\}$ is the unique minimal generating set of $\mathcal{PDF}_{n}$, and so $\rank(\mathcal{PDF}_{n}) =\idrank( \mathcal{PDF}_{n}) =2n$.
\end{theorem}

\proof By Propositions \ref{p10} and \ref{p13}, and Lemma \ref{l12}, it is enough to show that if $\alpha \in K_{n-1} \setminus (E_{n-1} \cup J_{n-1})$, then $\alpha \in \langle E_{n-1} \cup J_{n-1} \rangle$. Let  $\alpha$ be any element of $K_{n-1} \setminus (E_{n-1} \cup J_{n-1})$. Then it is clear that $\lvert \dom(\alpha) \rvert =n-1$, and so $\alpha$ is a partial injection. Then let $\beta =1_{\dom(\alpha)}$, which is an element of $E_{n-1}$. If we define the mapping $\gamma : [n]\rightarrow [n]$ by 
$$i\gamma =\left\{ \begin{array}{ll}
	i\alpha &\mbox{if } i\in \dom(\alpha),\\
	   i    &\mbox{if } i\notin \dom(\alpha),
\end{array}\right.$$
then it is clear that $\gamma \in J_{n-1}$ and $\alpha =\beta \gamma$, as required. \qed

\section{Fence-decreasing and fence-preserving partial transformations on $[n]$}

First, recall that $[n]^{o} \subseteq \fix(\alpha)$ for all $\alpha \in \mathcal{DF}_{n}$. Now we give a result, which was obtained earlier. 

\smallskip

\begin{corollary}\label{c17} \cite[Corollary 2.2]{F1}
	Let $\alpha\in \mathcal{F}_{n}$. Then $\im(\alpha)$ is an interval, i.e. $\im(\alpha) =[p,q]=\{ p,p+1, \ldots ,q\}$ for some $1\leq p\leq q\leq n$. \hfill $\square$
\end{corollary}

If $n$ is odd, then it follows from Corollary \ref{c17} that $\alpha \in \mathcal{DF}_{n}$ is fence-preserving if and only if $\alpha =1_{n}$. If $n$ is even, then $\alpha$ is fence-preserving if and only if $\alpha \in\{ 1_{n}, \lambda_{1n}\}$, as defined in (\ref{e1}).

\smallskip

Furthermore, let $[n]$ be a down-fence. Then, for each $j\in [n]^{o}$, we define the mappings $\mu_{1j} : [n] \rightarrow [n]$ (if $j\neq 1$) and  $\mu_{2j} : [n] \rightarrow [n]$ (if $j\neq n$) by
\begin{eqnarray*}
	i\mu_{1j} =\left\{\begin{array}{cl}
		i  & \mbox{if } i\in [n] \setminus \{j\},\\
		i-1 & \mbox{if } i=j,
	\end{array}	\right. &\mbox{and}&
	i\mu_{2j} =\left\{\begin{array}{cl}
		i  & \mbox{if } i\in [n] \setminus \{j\},\\
		i+1 & \mbox{if } i=j.
	\end{array}	\right.
\end{eqnarray*}
With these notations, one can show that 
$$J_{n,n-1}=\{ \mu_{1j} : j\in [n]^{o} \setminus \{1\} \} \cup \{ \mu_{2j} : j\in [n]^{o} \setminus \{n\} \}.$$
Moreover, if $[n]$ is a down-fence and $n$ is odd, then $\alpha\in \mathcal{DF}_{n}$ is fence-preserving if and only if $\alpha \in\{1_{n}, \mu_{1n}, \mu_{21}, \mu_{1n}\mu_{21}\}$, and if $n$ is even, then $\alpha$ is fence-preserving if and only if $\alpha \in\{1_{n}, \mu_{21}\}$. Thus, with these notations, if we let $\mathcal{CF}_{n}=\mathcal{PCF}_{n} \cap \mathcal{T}_{n}$, then we have the following corollary.

\smallskip

\begin{corollary}\label{c18}
	If $[n]$ is an up-fence, then $\mathcal{CF}_{n}=\left\{ \begin{array}{cl}
		\{1_{n}\} &\mbox{if $n$ is odd,}\\
		\{1_{n}, \lambda_{1n} \} &\mbox{if $n$ is even,}
	\end{array}\right.$ 
	and if $[n]$ is a down-fence, then $\mathcal{CF}_{n}=\left\{ \begin{array}{cl}
		\{1_{n}, \mu_{1n}, \mu_{21}, \mu_{1n}\mu_{21}\} &\mbox{if $n$ is odd,}\\
		\{1_{n}, \mu_{21}\} &\mbox{if $n$ is even.}
	\end{array}\right.$ \hfill $\square$
\end{corollary}

We consider fence-decreasing and fence-preserving partial transformations on the up-fence $[n]$, namely $\mathcal{PCF}_{n}$, from now on. Note that $\mathcal{PCF}_{n}$ is a submonoid of $\mathcal{PDF}_{n}$ and proper one whenever $n\geq 4$. Let $\alpha \in \mathcal{PCF}_{n}$ and let $j\in \dom(\alpha) \cap [n]^{e}$. Then we notice that $j+1 \notin \dom(\alpha)$ whenever $j\alpha =j-1$, and that $j-1\notin \dom(\alpha)$ whenever $j\alpha =j+1$. Moreover, by Lemma \ref{l9}, each element $\alpha\in \mathcal{PCF}_{n}$ satisfies the equation $\alpha^{3}= \alpha^{2}$, and by Proposition \ref{p11}, each transformation in $\mathcal{PDF}_{n}$ is uniquely determined by the kernel and the image. 

\smallskip

For each $0\leq r\leq n$, let 
$$L_{r}=L_{n,r}=\{ \alpha \in \mathcal{PCF}_{n} : \lvert \im(\alpha) \rvert =r\}.$$
Clearly, $L_{0} =\{ 0_{n}\}$ and $L_{n} =\{ 1_{n}\}$. Moreover, it is clear that if $n$ is odd, then $L_{n-1} =E_{n-1}$, and that if $n$ is even, then $L_{n-1} =E_{n-1} \cup\{ \lambda_{1n}, \pi\}$, as defined in (\ref{e1}), where 
\begin{eqnarray}\label{e2}
	\pi =\left(\begin{matrix}
	1&2&\cdots &n-2& n\\
	1&2&\cdots &n-2&n-1
\end{matrix}\right).
\end{eqnarray}
Since $\pi= 1_{\dom(\pi)} \lambda_{1n}$, we see that $\pi$ is not undecomposable in $\mathcal{PCF}(n,n-1)$. 

For each $1\leq r\leq n$, let $L^{*}_{r}  =L_{r} \cap K^{*}_{r}$, i.e.
$$L^{*}_{r} =\{ \alpha \in L_{r} : \sh(\alpha) \neq \emptyset \mbox{ and } j\alpha^{-1} =\{j-1,j+1\} \mbox{ for all } j\in \sh(\alpha) \}.$$
For every $n\geq 4$, notice that $L^{*}_{0} =L^{*}_{n-2} =L^{*}_{n-1} =L^{*}_{n} =\emptyset$, and that $L^{*}_{n-3} =\emptyset$ if $n$ is odd.

\smallskip 

As in Lemma \ref{l12}, we also have the following lemma.

\smallskip 

\begin{lemma} \label{l19}
	\begin{itemize}
		\item [$(i)$] Let $1\leq r\leq n$. Then each element of $E(L_{r})$ is undecomposable in $\mathcal{PCF}(n,r)$. 
		\item [$(ii)$] Let $1\leq r\leq n-4$ (if $n>4$). Then each element of $L^{*}_{r}$ is undecomposable in $\mathcal{PCF}(n,r)$. Moreover, if $n$ is even, then each element of $L^{*}_{n-3}$ is undecomposable in $\mathcal{PCF}(n,n-3)$.  
	\end{itemize}
\end{lemma}

\proof The proof of $(i)$ is almost the same with the proof of Lemma \ref{l5}. Moreover, since $L^{*}_{r} \subseteq K^{*}_{r}$, and since $\mathcal{PCF}(n,r)$ is a subsemigroup of $\mathcal{PDF}(n,r)$, it follows from Lemma \ref{l12} that each element of $L^{*}_{r}$ is undecomposable in $\mathcal{PDF}(n,r)$, and so in $\mathcal{PCF}(n,r)$. \qed

Let $\alpha \in \mathcal{PCF}_{n}$, and let
\begin{eqnarray*}
	&&B_{\alpha} =\{\, j\in \dom(\alpha)\cap [n]^{e} : j\alpha =j-1 \mbox{ and } j-1 \notin \dom(\alpha) \mbox{ or}\\
	&&\hspace{49mm} j\alpha =j+1 \mbox{ and } j+1 \notin \dom(\alpha) \mbox{ if } j\neq n\}.
\end{eqnarray*}
Then $\alpha \in \mathcal{PCF}_{n}$ is idempotent if and only if $B_{\alpha} =\emptyset$. Moreover, if $\alpha \in \mathcal{PCF}_{n}$ is not idempotent, then it is clear that $j-1,j+1 \notin \dom(\alpha)$ for every $j\in B_{\alpha}$. Instead of either $i\notin \dom(\alpha)$ or $i\in \dom(\alpha)$ and $i\alpha \in \{i-1,i\}$ ($i\alpha \in \{i,i+1\}$), we simply write $i\alpha \neq i+1$ ($i\alpha \neq i-1$) in the following proof. Moreover, for each $\alpha \in \mathcal{PCF}_{n} \setminus E(\mathcal{PCF}_{n})$, we let
\begin{eqnarray*}
	B^{1}_{\alpha} =\{\, j,j+2\in B_{\alpha} : j\alpha =j+1= (j+2)\alpha \} &\mbox{and}& B^{2}_{\alpha} =B_{\alpha} \setminus B^{1}_{\alpha}.
\end{eqnarray*}
Let $\alpha \in L_{r}$. Then it is clear that $\alpha \in L_{r}^{*}$ if and only if $B^{2}_{\alpha} =\emptyset$.

\smallskip

\begin{proposition} \label{p20}
	Let $1\leq r\leq n-2$. Then $L_{r} \subseteq \langle E(L_{r}) \cup L_{r}^{*} \rangle$.
\end{proposition}

\proof Let $\alpha \in L_{r}\setminus (E(L_{r}) \cup L_{r}^{*})$. Then it is clear that $B^{2}_{\alpha} \neq \emptyset$. Moreover, for each $j\in B^{2}_{\alpha}$, either $j\alpha =j-1$ and $(j-2)\alpha \neq j-1$ (if $j\geq 4$) or $j\alpha =j+1$ and $(j+2)\alpha \neq j+1$ (if $j\leq n-2$). That is, $\lvert (j\alpha) \alpha^{-1} \rvert =1$ for every $j\in B^{2}_{\alpha}$. If $B^{1}_{\alpha} \neq \emptyset$, then we consider the mappings $\beta : \dom(\alpha) \rightarrow [n]$ and $\gamma : \im(\beta) \rightarrow [n]$ by
\begin{eqnarray*}
	i\beta =\left\{\begin{array}{cl}
		i\alpha &\mbox{if } i\in \dom(\alpha)\setminus B^{2}_{\alpha},\\
		i       &\mbox{if } i\in B^{2}_{\alpha}
	\end{array}	\right. &\mbox{and}&
	i\gamma =\left\{\begin{array}{cl}
		i		&\mbox{if } i\in \im(\beta) \setminus B^{2}_{\alpha},\\
		i\alpha &\mbox{if } i\in B^{2}_{\alpha}.
	\end{array}	\right.
\end{eqnarray*} 
Then clearly, $\beta\in L_{r}^{*}$,\, $\gamma \in L_{r}$ and $\alpha =\beta \gamma$. If $B^{1}_{\alpha} =\emptyset$, then we let $\gamma =\alpha$. In both cases, we notice that $\gamma$ is injective, and that $B_{\gamma} =B^{2}_{\gamma} =B^{2}_{\alpha}$.  

\smallskip

Since $\alpha$ is not idempotent, notice that $B_{\gamma}= B^{2}_{\gamma} \neq \emptyset$. Let $j=\min \{ i\in B_{\gamma} : i\gamma =i-1 \}$. Then $j-1,j+1\notin \dom(\gamma)$ and $(j+2)\gamma \in \{j+1 ,j+2, j+3\}$ if $j+2\in \dom(\gamma)$. Next we define the mappings $\mu : \dom(\gamma) \rightarrow [n]$ and $\delta : \im(\mu) \cup \{j-1\} \rightarrow [n]$ by
\begin{eqnarray*}
		i\mu =\left\{\begin{array}{cl}
		i\gamma &\mbox{if } i\in \dom(\gamma) \cap [1,j-1],\\
		i       &\mbox{if } i\in \dom(\gamma) \cap [j,n]
	\end{array}	\right.&\mbox{and}&
	i\delta =\left\{\begin{array}{cl}
		i		&\mbox{if } i\in \dom(\delta) \cap [1,j-1],\\
		i\gamma	&\mbox{if } i\in \dom(\delta) \cap [j,n].
	\end{array}	\right.
\end{eqnarray*} 
Then it is clear that $\mu\in E(L_{r})$,\, $\delta\in L_{r}$,\, $\lvert B_{\delta} \rvert= \lvert B_{\gamma} \rvert-1$, and that $\gamma =\mu \delta$. Let $j=\min \{ i\in B_{\gamma} : i\gamma =i+1 \}$. Then $(j+2)\gamma \neq j+1$ if $j+2\in \dom(\gamma)$ since $j\notin B^{1}_{\gamma}$. Now we define the mappings $\theta: \dom(\gamma) \rightarrow [n]$ and $\lambda : \im(\gamma) \cup \{j\} \rightarrow [n]$ by
\begin{eqnarray*}
	i\theta =\left\{\begin{array}{cl}
		i\gamma &\mbox{if } i\in \dom(\gamma) \setminus \{j\}\\
		j       &\mbox{if } i=j.
	\end{array}	\right.&\mbox{and}&
	i\lambda =\left\{\begin{array}{cl}
		i		&\mbox{if } i\in \im(\gamma),\\
		j+1 &\mbox{if } i=j.
	\end{array}	\right.
\end{eqnarray*} 
Then it is clear that $\theta \in L_{r}$,\, $\lambda \in E(L_{r})$,\, $\lvert B_{\theta} \rvert= \lvert B_{\gamma} \rvert-1$, and that $\gamma =\theta \lambda$.

\smallskip

Therefore, by induction first on $\lvert \{i\in B_{\gamma} : i\gamma =i-1 \}\rvert$, and then on $\lvert \{i\in B_{\gamma} : i\gamma =i+1 \}\rvert$, the proof is now completed. \qed

\begin{example} If $\alpha= \left(\begin{matrix}
	2&4&6&8&10&12\\
	3&3&5&7&11&13
\end{matrix}\right) \in L_{13,5}$, then $B^{1}_{\alpha} =\{2,4\}$ and $B^{2}_{\alpha} =\{6,8,10,12\}$. By applying the first factorization, given in the proof above, we have $\alpha=\beta \gamma$ where 
$\beta= \left(\begin{matrix}
	2&4&6&8&10&12\\
	3&3&6&8&10&12
\end{matrix}\right) \in L^{*}_{5}$ 
and 
$\gamma =\left(\begin{matrix}
	3&6&8&10&12\\
	3&5&7&11&13
\end{matrix}\right) \in L_{5}$. 
Moreover, by applying the other factorizations, we have
$$\begin{array}{rcl}\gamma &=& \left(\begin{matrix}
	3&6&8&10&12\\
	3&6&8&10&12
\end{matrix}\right) \left(\begin{matrix}
	3&5&6&8&10&12\\
	3&5&5&8&10&12
\end{matrix}\right) \left(\begin{matrix}
3&5&7&8&10&12\\
3&5&7&7&10&12
\end{matrix}\right) \vspace*{1mm}\\
&&\left(\begin{matrix}
	3&5&7&10&12&13\\
	3&5&7&10&13&13
\end{matrix}\right)\left(\begin{matrix}
3&5&7&10&11&13\\
3&5&7&11&11&13
\end{matrix}\right), \end{array}$$
and each factor is an idempotent in $L_{5}$.
\end{example} 

Let $1\leq r\leq n-1$, let $\alpha \in \mathcal{PCF}(n,r)$, and let 
$$C_{\alpha} =\{ j\in \dom(\alpha) \cap [n]^{e} : \lvert (j\alpha)\alpha^{-1} \rvert \geq 2\}.$$
Notice that either $j+1\notin \dom(\alpha)$ or $j-1\notin \dom(\alpha)$ for each $j\in C_{\alpha}$ if $C_{\alpha} \neq \emptyset$. 

\smallskip

To show that $L_{r}\subseteq \langle L_{r+1}\rangle$ for $1\leq r\leq n-3$, by Proposition \ref{p20}, it is sufficient to prove the following proposition. 

\smallskip

\begin{proposition} \label{p21}
	Let $1\leq r\leq n-3$. Then $E(L_{r}) \cup L_{r}^{*} \subseteq \langle E(L_{r+1}) \cup L_{r+1}^{*}\rangle$.
\end{proposition}

\proof Let  $\alpha \in E(L_{r})$. Since, by Corollary \ref{c18}, every element of $E(L_{r})$ is strictly partial, i.e. not full, and since each partial identity in $E(L_{r})$ can be written as product of two partial identities in $E(L_{r+1})$, it is enough to consider idempotents of $E(L_{r})$, which are not partial identities. Let $\alpha \in E(L_{r})$ be a such an element. Clearly, $C_{\alpha} \neq \emptyset$. First suppose that $C_{\alpha} =\{ j\}$. Then $\lvert  \dom(\alpha) \rvert =r+1$, and either $(j-1)\alpha =j-1= j\alpha$ and $j+1\notin \dom(\alpha)$, or $(j+1)\alpha =j+1= j\alpha$ and $j-1\notin \dom(\alpha)$. In the both cases, $[n]\setminus (\dom(\alpha) \cup \{j-1, j+1\}) \neq \emptyset$. If we let  $k\in [n]\setminus (\dom(\alpha) \cup \{j-1, j+1\})$, and define the mapping $\beta :\dom(\alpha) \cup \{k\} \rightarrow [n]$ by
\begin{eqnarray*}
	i\beta =\left\{\begin{array}{cl}
		i\alpha		&\mbox{if } i\in \dom(\alpha),\\
		k		&\mbox{if } i=k,
	\end{array}	\right.
\end{eqnarray*} 
then it is clear that $1_{\dom(\alpha)},\, \beta \in E(L_{r+1})$ with $\alpha= 1_{\dom(\alpha)} \beta$. 

\smallskip

Suppose that $\lvert C_{\alpha} \rvert \geq 2$, and let $j= \min(C_{\alpha})$. First notice that $r+2\leq \lvert  \dom(\alpha) \rvert \leq n-2$. If we define the mappings $\gamma : \dom(\alpha) \rightarrow [n]$ and $\delta : \im(\alpha) \cup \{j\} \rightarrow [n]$ by
\begin{eqnarray*}
	i\gamma =\left\{\begin{array}{cl}
		i\alpha &\mbox{if } i\in \dom(\alpha)\setminus \{j\},\\
		j       &\mbox{if } i=j
	\end{array}	\right. &\mbox{and}&
	i\delta =\left\{\begin{array}{cl}
		i		&\mbox{if } i\in \im(\alpha),\\
		j\alpha &\mbox{if } i=j,
	\end{array}	\right.
\end{eqnarray*} 
then it is clear that $\gamma\in E(L_{r+1})$,\, $\delta \in E(L_{r})$, and $\alpha =\gamma \delta$. Thus, since $\lvert C_{\delta} \rvert=1$, from the previous case, we conclude that $E(L_{r}) \subseteq \langle E(L_{r+1})\rangle$. 

\smallskip

Let $\alpha \in L_{r}^{*}$. Thus, there exists $2\leq j\leq n-2$ such that $j\alpha =j+1= (j+2)\alpha$ and $j+1\notin \dom(\alpha)$, i.e. $j,\, j+2 \in B^{1}_{\alpha}$. If we define the mappings $\lambda : \dom(\alpha) \rightarrow [n]$ and $\mu : \im(\alpha) \cup \{j,j+2\} \rightarrow [n]$ by
\begin{eqnarray*}
	i\lambda =\left\{\begin{array}{cl}
		i\alpha &\mbox{if } i\in \dom(\alpha)\setminus \{j,j+2\},\\
		i       &\mbox{if } i\in \{j,j+2\}
	\end{array}	\right. &\mbox{and}&
	i\mu =\left\{\begin{array}{cl}
		i	&\mbox{if } i\in \im(\alpha),\\
		j+1 &\mbox{if } i\in \{j,j+2\},
	\end{array}	\right.
\end{eqnarray*} 
then it is clear that $\lambda\in L_{r+1}$,\, $\mu \in E(L_{r})$ and $\alpha =\lambda \mu$. By Proposition \ref{p20}, the proof is now completed. \qed

Thus, by Lemma \ref{l19}, and Propositions \ref{p20} and \ref{p21}, we have the following immediate theorem.

\smallskip

\begin{theorem}\label{t22}
	Let $1\leq r\leq n-2$. Then $E(L_{r}) \cup L_{r}^{*}$ is the unique minimal generating set of $\mathcal{PCF}(n,r)$. \hfill $\square$
\end{theorem}

Let $f_{n}= \lvert E(\mathcal{PCF}_{n}) \rvert$, and for $0\leq r\leq n$, let $f_{n,r} =\lvert E(L_{n,r}) \rvert$. Since $f_{n,0}=f_{n,n}=1$ and 
$f_{n,1}=e_{n,1}=\left\{ \begin{array}{ll}
	5m & \mbox{if $n$ is odd,}\vspace{1mm} \\
	5m-2 & \mbox{if $n$ is even,}
\end{array}\right.$
where $m= \lfloor \frac{n}{2}\rfloor$, we determine the other values of $f_{n,r}$ in the following lemma and in Table 3. First, recall that $\mathcal{PCF}_{n} =\mathcal{PDF}_{n}$ for all $1\leq n\leq 3$.

\smallskip

\begin{proposition} \label{p23} 
	Let $n\geq 4$ and let $2\leq r\leq n-2$. Then we have 
	$$f_{n,r} =\left\{ \begin{array}{ll}
		f_{n-1,r}+ f_{n-2,r-1} +f_{n-2,r-2} + f_{n-3,r-1} & \mbox{if $n$ is odd,}\vspace{1mm} \\
		f_{n-1,r}+ f_{n-1,r-1} +f_{n-3,r-1} +f_{n-3,r-2} +f_{n-4,r-1} & \mbox{if $n$ is even}
	\end{array}\right.$$
	with the assumption that $f_{k,k+1}=0$ for every $k\geq 0$. Moreover, we have 
	\begin{eqnarray*}
		f_{n,n-1} =\left\{ \begin{array}{ll}
		n & \mbox{if $n$ is odd,}\vspace{1mm} \\
		n+1 & \mbox{if $n$ is even} 
	\end{array}\right. &\mbox{and}& 
	f_{n} =\left\{ \begin{array}{ll}
		2\cdot 5^{m} & \mbox{if $n$ is odd,}\vspace{1mm} \\
		5^{m} & \mbox{if $n$ is even,}
	\end{array}\right.
	\end{eqnarray*}
	where $m=\lfloor \frac{n}{2}\rfloor$.
\end{proposition}

\proof Since $E(L_{n-1}) =L_{n-1} =E_{n-1}$ if $n$ is odd, and since $E(L_{n-1}) =E_{n-1} \cup\{ \lambda_{1n}\}$ if $n$ is even, we have 
$f_{n,n-1} =\left\{ \begin{array}{ll}
	n & \mbox{if $n$ is odd,}\vspace{1mm} \\
	n+1 & \mbox{if $n$ is even.} 
\end{array}\right.$

Suppose $n$ is odd. Since $n\alpha=n$ whenever $n\in \dom(\alpha)$, it follows that $E(L_{n,r})$ is a disjoint union of the following sets:
\begin{eqnarray*}
	&&F^{1}_{n,r} =\{ \alpha \in E(L_{n,r}) : n\notin \dom(\alpha) \},\\
	&&F^{2}_{n,r} =\{ \alpha \in E(L_{n,r}) : n\alpha^{-1} =\{n\} \mbox{ and } n-1\notin \dom(\alpha)\},\\
	&&F^{3}_{n,r} =\{ \alpha \in E(L_{n,r}) : n\alpha^{-1} =\{n\} \mbox{ and } (n-1)\alpha^{-1} =\{n-1\}\}, \mbox{ and}\\
	&&F^{4}_{n,r} =\{ \alpha \in E(L_{n,r}) : n\alpha^{-1} =\{n-1,n\}\}. 
\end{eqnarray*}
Since it can be easily seen that
$\lvert F^{1}_{n,r} \rvert= \lvert E(L_{n-1,r}) \rvert$,\,  $\lvert F^{2}_{n,r} \rvert= \lvert E(L_{n-2,r-1})\rvert$,\, $\lvert F^{3}_{n,r} \rvert= \lvert E(L_{n-2,r-2}) \rvert$ and $\lvert F^{4}_{n,r} \rvert= \lvert E(L_{n-3,r-1}) \rvert$, 
we conclude that $f_{n,r} =f_{n-1,r}+ f_{n-2,r-1} +f_{n-2,r-2} + f_{n-3,r-1}$. 

\smallskip

Suppose $n$ is even. Since $n\alpha=n$ or $n\alpha=n-1$  whenever $n\in \dom(\alpha)$, it follows that $E(L_{n,r})$ is a disjoint union of the following sets:
\begin{eqnarray*}
	&&F^{5}_{n,r} =\{ \alpha \in E(L_{n,r}) : n\notin \dom(\alpha) \},\\
	&&F^{6}_{n,r} =\{ \alpha \in E(L_{n,r}) : n\alpha^{-1} =\{n\}\, \},\\
	&&F^{7}_{n,r} =\{ \alpha \in E(L_{n,r}) : (n-1)\alpha^{-1} =\{n-1,n\} \mbox{ and }n-2\notin \dom(\alpha)\}, \\
	&&F^{8}_{n,r} =\{ \alpha \in E(L_{n,r}) : (n-1)\alpha^{-1} =\{n-1,n\}\mbox{ and }(n-2)\alpha =n-2\},\mbox{ and}\\
	&&F^{9}_{n,r} =\{ \alpha \in E(L_{n,r}) : (n-1)\alpha^{-1} =\{n-2,n-1,n\}\,\}. 
\end{eqnarray*}
Since it can be easily seen that
$\lvert F^{5}_{n,r} \rvert= \lvert E(L_{n-1,r}) \rvert$,\,   $\lvert F^{6}_{n,r} \rvert= \lvert E(L_{n-1,r-1}) \rvert$,\,   $\lvert F^{7}_{n,r} \rvert=  \lvert E(L_{n-3,r-1}) \rvert$,\, $\lvert F^{8}_{n,r} \rvert= \lvert E(L_{n-3,r-2}) \rvert$, and $\lvert F^{9}_{n,r} \rvert= \lvert E(L_{n-4,r-1}) \rvert$ (notice that $E(L_{n-4,r-1}) =\emptyset$ when ever $r=n-2$), we conclude that $f_{n,r} =f_{n-1,r}+ f_{n-1,r-1} +f_{n-3,r-1} +f_{n-3,r-2} +f_{n-4,r-1}$. 

\smallskip

Moreover, if $n$ is odd, then $f_{n}= f_{n-1}+ 2f_{n-2} +f_{n-3}$ since $E(\mathcal{PCF}_{n})$ is a disjoint union of the following sets:
\begin{eqnarray*}
	&&\{ \alpha \in E(\mathcal{PCF}_{n}) : n\notin \dom(\alpha) \},\\ 
	&&\{ \alpha \in E\mathcal{PCF}_{n}) : n\alpha^{-1} =\{n\} \mbox{ and } n-1\notin \dom(\alpha)\},\\
	&&\{ \alpha \in E\mathcal{PCF}_{n}) : n\alpha^{-1} =\{n\} \mbox{ and } (n-1)\alpha =n-1\}, \mbox{ and}\\
	&&\{ \alpha \in E(\mathcal{PCF}_{n}) : n\alpha^{-1} =\{n-1,n\}\,\}, 
\end{eqnarray*}
and if $n$ is even, then $f_{n}= 2f_{n-1}+ 2f_{n-3}+ f_{n-4}$ since $E(\mathcal{PDF}_{n})$ is a disjoint union of the following sets:
\begin{eqnarray*}
	&&\{ \alpha \in E(\mathcal{PCF}_{n}) : n\notin \dom(\alpha) \},\\
	&&\{ \alpha \in E(\mathcal{PCF}_{n}) : n\alpha^{-1} =\{n\}\,\}\,\\
	&&\{ \alpha \in E(\mathcal{PCF}_{n}) : (n-1)\alpha^{-1} =\{n-1,n\} \mbox{ and } n-2\notin \dom(\alpha)\}, \\
	&&\{ \alpha \in E(\mathcal{PCF}_{n}) : (n-1)\alpha^{-1} =\{n-1,n\}  \mbox{ and } (n-2)\alpha^{-1} =\{n-2\}\}, \\
	&&\{ \alpha \in E(\mathcal{PCF}_{n}) : (n-1)\alpha^{-1} =\{n-2, n-1,n\}\}. 
\end{eqnarray*}
Therefore, one can easily show inductively on $n$ that $f_{n} =\left\{ \begin{array}{ll}
	2\cdot 5^{m} & \mbox{if $n$ is odd,}\vspace{1mm} \\
	5^{m} & \mbox{if $n$ is even,}
\end{array}\right.$ where $m=\lfloor \frac{n}{2}\rfloor$. \qed

\begin{table}[!ht]	
	\begin{center} \label{Tab:card3}
		\begin{tabular}{c|cccccccc|c}	
			n$\backslash$r  & 0 & 1 & 2 & 3  & 4  & 5 & 6 & 7 & $\lvert E(\mathcal{PCF}_{n})\rvert$ \\	\hline \vspace*{-5mm}\\
			1  & 1 &   1 &     &     &    &    &   &   &    2  \\
			\hline \vspace*{-5mm}\\
			2  & 1 &   3 &   1 &     &    &    &   &   &    5  \\
			\hline \vspace*{-5mm}\\
			3  & 1 &   5 &   3 &   1 &    &    &   &   &   10  \\
			\hline \vspace*{-5mm}\\
			4  & 1 &  8 &  10 &  5 &  1 &    &   &   &   25  \\
			\hline \vspace*{-5mm}\\
			5  & 1 &  10 &  19 &  14 & 5&  1 &   &   &   50  \\
			\hline \vspace*{-5mm}\\
			6  & 1 & 13 & 38 & 42 & 23 & 7  & 1 &    &  125  \\
			\hline \vspace*{-5mm}\\
			7  & 1 & 15 & 57 & 81 & 61 & 27 & 7 & 1 & 250 \\
			
		\end{tabular}
		\caption{$\lvert E(L_{n,r}) \rvert$ and $\lvert E(\mathcal{PCF}_{n})\rvert$.}
	\end{center}
\end{table}

\smallskip

To find minimal generating sets of $\mathcal{PCF}(n,n-1)$ and $\mathcal{PCF}_{n}$, we determine some crucial elements in $L_{n-2}$. Since $E_{n-2} \subseteq \langle E_{n-1}\rangle$, we just consider elements of $L_{n-2} \setminus E_{n-2}$. If $\alpha \in L_{n-2} \setminus E_{n-2}$, then $(\dom(\alpha) \cap [n]^{e}) \setminus \fix(\alpha) \neq \emptyset$. For each $j\in [n]^{e}\setminus \{n\}$, we define the mappings $\xi_{1j} : [n]\setminus \{j+1\} \rightarrow [n]$ and  $\xi_{2j} : [n]\setminus \{j-1\} \rightarrow [n]$ by
\begin{eqnarray*}
	i\xi_{1j} =\left\{\begin{array}{cl}
		i  & \mbox{if } i\in [n]\setminus \{j,j+1\},\\
		j-1 & \mbox{if } i=j
	\end{array}	\right. &\mbox{and}&
	i\xi_{2j} =\left\{\begin{array}{cl}
		i  & \mbox{if } i\in [n]\setminus \{j,j-1\},\\
		j+1 & \mbox{if } i=j.
	\end{array}	\right.
\end{eqnarray*}
Moreover, if we let 
$$L^{\pm}_{n-2} =\{ \xi_{1j},\, \xi_{2j} : j\in [n]^{e}\setminus \{n\} \},$$
then it is a routine matter to check that $L_{n-2}= E_{n-2}\cup L^{\pm}_{n-2}$ if $n$ is odd. However, it is not that simple if $n$ is even. Let $\alpha \in L_{n-2} \setminus E_{n-2}$. If $n\alpha =n$, then clearly, $\alpha \in L^{\pm}_{n-2}$. If $n\alpha =n-1$, then there exists a subset $Y$ of $[n-2]$ with $n-3$ elements such that either $\dom(\alpha)=Y\cup \{n\}$ or $\dom(\alpha)=Y\cup \{n-1,n\}$. Moreover, since we have $\alpha =1_{\dom(\alpha)} \lambda_{1n}$ (as defined in (\ref{e1})) if $n\alpha =n-1$, we obtain $L_{n-2} \subseteq \langle E_{n-2}\cup L^{\pm}_{n-2} \cup \{ \lambda_{1n} \} \rangle$. Hence we need to check the elements of $L^{\pm}_{n-2}$ if they are undecomposable in $\mathcal{PCF}(n,n-1)$. 

\smallskip

\begin{lemma} \label{l24}
	Each element of $L^{\pm}_{n-2}$ is undecomposable in $\mathcal{PCF}(n,n-1)$, and so in $\mathcal{PCF}_{n}$.  
\end{lemma}

\proof Let $j\in [n]^{e}\setminus \{n\}$, and suppose that $\xi_{1j} =\beta \gamma$ for some $\beta, \gamma \in \mathcal{PCF}(n,n-1)$. Then $\fix(\xi_{1j}) =[n] \setminus \{j,j+1\}$ is a subset of both $\fix(\beta)$ and $\fix(\gamma)$, and so either $j\beta =j-1$ or $j\beta =j$. Thus, either $\xi_{1j} =\beta$ in the first case, or $\xi_{1j} =\gamma$ in the second case. One can similarly show that $\xi_{2j}$ is undecomposable in $\mathcal{PCF}(n,n-1)$.  Since $L_{n}= \{1_{n}\}$, the proof is completed. \qed

We notice that $\lvert L^{\pm}_{n-2} \rvert=\left\{ \begin{array}{cl}
	n-1& \mbox{if $n$ is odd},\\
	n-2& \mbox{if $n$ is even},
	\end{array} \right.$
and remind that $L^{*}_{r} =\emptyset$ for each $n-2\leq r\leq n$. Before giving one of the main theorems of this section, we let
$$G_{n-1} =\left\{ \begin{array}{ll}
	E_{n-1}	\cup L^{\pm}_{n-2}			& \mbox{if $n$ is odd},\\
	E_{n-1}	\cup L^{\pm}_{n-2} \cup\{ \lambda_{1n}\}	& \mbox{if $n$ is even}.
\end{array} \right.$$
Then, by Propositions \ref{p20}, \ref{p21} and \ref{p23}, and Lemma \ref{l24}, we have the following theorem.

\begin{theorem}\label{t25}
	$G_{n-1}$ is the unique minimal generating set of $\mathcal{PCF}(n,n-1)$, and so $G_{n-1} \cup\{ 1_{n}\}$ is the unique minimal generating set of $\mathcal{PCF}_{n}$. Furthermore, 
	$$\rank(\mathcal{PCF}(n-1))= \idrank(\mathcal{PCF}(n-1)) = 2n-1$$
	whatever the parity of $n$ is, and so
	$\rank(\mathcal{PCF}_{n})= \idrank(\mathcal{PCF}_{n})= 2n$
	whatever the parity of $n$ is. \hfill $\square$
\end{theorem}

\noindent\textbf{Research ethics:} Not applicable (this research does not involve human participants or animals).

\noindent\textbf{Informed consent:} Not applicable.

\noindent\textbf{Use of Large Language Models, AI and Machine Learning Tools}: None declared.

\noindent\textbf{Consent to participate:} Not applicable.

\noindent\textbf{Consent for publication:} Not applicable.

\noindent\textbf{Author contributions:} The author is solely responsible for the work.

\noindent\textbf{Conflict of interest:} The authors declares that there are no competing interests.

\noindent\textbf{Research funding:} The first two authors acknowledges the support by the Scientific and Technological Research Council of Turkey (TUBITAK) under the Grant Number 123F463.

\noindent\textbf{Data availability:} Not applicable.

\end{document}